\documentclass[11pt]{article}
\usepackage[latin1]{inputenc}
\usepackage{enumitem, amsthm, multirow}
\usepackage{amsmath, framed,wasysym}
\usepackage{amssymb,ulem}

\usepackage[english]{babel}
\usepackage{cancel}
\usepackage{color}
\definecolor{Blue}{rgb}{0,0,1}
\definecolor{Red}{rgb}{1,0,0}

\usepackage[curve,matrix,arrow,cmtip]{xy}
\NoComputerModernTips

\newdir^{ (}{{}*!/-3pt/\dir^{(}}    
\newdir_{ (}{{}*!/-3pt/\dir_{(}}    

\usepackage[scale=0.7,hmarginratio=1:1,vmarginratio=2:3]{geometry}

\newcounter{themargin}
\long\def\forget#1{}

\theoremstyle{plain} 
\newtheorem{Thm}{Theorem}[section]

\newtheorem{Prop}[Thm]{Proposition}

\newtheorem{Lem}[Thm]{Lemma}

\newtheorem{Cor}[Thm]{Corollary}

\theoremstyle{definition} 
\newtheorem{Def}[Thm]{Definition}

\newtheorem{Ex}[Thm]{Example}

\theoremstyle{remark} 
\newtheorem{Rem}[Thm]{Remark}

\newtheorem{Conj}[Thm]{Conjecture}

\forget{\newif\ifnormalesBeweisEnde

  {\vskip 0.3ex plus 0.5ex minus 0ex \pagebreak[1]
   \global\normalesBeweisEndetrue
   \trivlist
   \item[\hskip\labelsep {\textsc{Proof} \rm #1}:]}%
  {\ifnormalesBeweisEnde \EndOfProof \fi
   \endtrivlist
   \vskip 1ex plus 1ex minus 0ex \pagebreak[2]}
  {\vskip 0.3ex plus 0.5ex minus 0ex \pagebreak[1]
   \global\normalesBeweisEndetrue
   \trivlist
   \item[\hskip\labelsep \textsc{Proof}:]}%
  {\ifnormalesBeweisEnde \EndOfProof \fi
   \endtrivlist
   \vskip 1ex plus 1ex minus 0ex \pagebreak[2]}
\def\EndOfProof{\hskip .5em \vrule width 1.0ex height 1.0ex depth 0.3ex}
}

\usepackage{latexsym}
\usepackage{amsmath}
\usepackage{amstext}
\usepackage{amssymb}
\usepackage{amsbsy}
\usepackage{mathrsfs}

\newbox\dottobox
\setbox\dottobox=\hbox{$\longrightarrow$}
\setbox\dottobox=\hbox to\wd\dottobox{\hss$
\UseComputerModernTips\xymatrix@C=.5cm{\ar@{.>}[r]&\\}
                                      $\hss}

\newbox\leftdottobox
\setbox\leftdottobox=\hbox{$\longleftarrow$}
\setbox\leftdottobox=\hbox to\wd\leftdottobox{\hss$
\UseComputerModernTips\xymatrix@C=.5cm{\ar@{<.}[r]&\\}
                                      $\hss}

\newbox\dotintobox
\setbox\dotintobox=\hbox{$\longrightarrow$}
\setbox\dotintobox=\hbox to\wd\dotintobox{\hss$
\UseComputerModernTips\xymatrix@C=.5cm{\ar@{^{ (}.>}[r]&\\}
                                      $\hss}

\newcommand{\Tempnewpage}
{\newpage}

\DeclareMathOperator{\GL}{GL}

\begin{document}
\title{Functional Hecke algebras and simple 
Bernstein blocks of a $p$-adic $\GL_n$ in non-defining characteristic}
\author{David-Alexandre Guiraud\footnote{
Interdisciplinary Center for Scientific Computing,
Heidelberg University, Germany david.guiraud@iwr.uni-heidelberg.de}$\;\,$\footnote{This research was 
conducted while the author received funding from the Graduiertenkolleg Heidelberg (LGWG scholarship).}
}
\maketitle
\vspace{2cm}
\begin{abstract}
Let $G_{n}=\operatorname{GL}_{n}(F)$, where $F$ is a non-archimedean local 
field with
residue characteristic $p$ and where $n=2k$ is even. In this article, we investigate a question occurring in the decomposition of 
the category of $\ell$-modular smooth representations of $G_n$ into Bernstein blocks (where $\ell\neq p$). The easiest block not investigated in 
\cite{guiraud} is the one defined by the standard parabolic subgroup with Levi factor $M=\GL_k(F) \times \GL_k(F)$ and by
an $M$-representation of the form $\pi_0 \boxtimes \pi_0$ with $\pi_0$ a supercuspidal $\GL_k(F)$-representation.
This block is Morita equivalent to a Hecke algebra which we can describe as a twisted tensor product of a finite Hecke algebra
(i. e. a Hecke algebra occurring in the representation theory of the finite group $\GL_k(p^{\alpha})$ in non-defining characteristic $\ell$) 
and the group ring of $\mathbb{Z}^2$. This enables us to describe how a conjectured connection between finite Hecke algebras 
(which is similar to a connection postulated by Brou\'{e} in \cite{Broue}) would lead to an equivalence between the described block and
the unipotent block of $\operatorname{GL}_2(F^k)$, where $F^k$ is the unramified extension of degree $k$ over $F$.
\end{abstract}
\section{Introduction}
Let $F$ be a non-archimedean local field with ring of integers $\mathcal{O}_F$, uniformizer $\varpi_F$ and finite residue field $k_F \cong \mathbb{F}_{q} = \mathbb{F}_{p^{\alpha}}$ for some
prime $p$. Moreover, let $R$ be an algebraically closed field of positive characteristic
$\ell \neq p$ such that $R$ arises as residue field in some $\ell$-modular system 
$(R, \mathcal{O}_K, K)$. Continuing the exposition in \cite{guiraud}, we consider the category $\mathfrak{R}_R(G)$ of smooth
$R$-valued representations of the group $G = G_n = \GL_n(F)$. In particular, recall (e.g. from Theorem 2.22 of \cite{guiraud}) that 
the level-$0$ part of $\mathfrak{R}_R(G)$ decomposes as a direct sum of Bernstein-blocks, each of the form
\begin{equation*}
\mathfrak{R}_R^{[M, \pi]_G}(G) = \bigg\{ V\in \mathfrak{R}_R(G)\,\bigg|\, \substack{\text{ any irreducible JH-constituent of $V$ is a JH-constituent
}\\ \text{of $\operatorname{i}^G_{N\subset P}(\sigma)$ for some $(N,\sigma) \in [M, \pi]_G$}}\bigg\},
\end{equation*}
where $[M, \pi]_G$ denotes a $G$-equivalence class of supercuspidal pairs and $\operatorname{i}^G_{N\subset P}$ denotes 
parabolic induction. So, in particular, each $M$ is a Levi subgroup of $G$ and
$\pi$ is an irreducible supercuspidal $M$-representation. In this paper, we will study simple Bernstein-blocks of the
smallest non-trivial form, i. e. where $M \cong \GL_k(F) \times \GL_k(F)$ and $\pi \cong \pi_0 \boxtimes \pi_0$, where
$n= 2k$ and $\pi_0$ is an irreducible supercuspidal level-$0$ $\GL_k(F)$-representation\footnote{As \cite{guiraud} explained how to 
reduce from arbitrary blocks to simple blocks, this is the logically next case to consider, from the point of view of \cite{guiraud}.}.

Now, if $\mathcal{V}\in \mathfrak{R}_R^{[M, \pi]_G}(G)$ is a pro-generator, we get an equivalence 
\begin{equation*}
\mathfrak{R}_R^{[M, \pi]_G}(G) \underset{\text{Morita}}{\cong} \operatorname{End}_G(\mathcal{V}).
\end{equation*}
As $\pi_0$ is level-$0$, it is of the form $\operatorname{ind}^{M}_{Z\GL_k(\mathcal{O}_F)\times Z\GL_k(\mathcal{O}_F)}(\rho)$, 
where $\rho = \rho_0 \boxtimes \rho_0$ and $\rho_0$
it is inflated from an irreducible supercuspidal representation of the finite group $\mathcal{G}_k = \GL_k(q)$ and where
$Z$ denotes the center of $\GL_k(F)$. We will not
distinguish between $\rho_0$ (resp. $\rho$) as a representation of $\GL_k(\mathcal{O}_F)$ and of $\mathcal{G}_k$
(resp. of $\GL_k(\mathcal{O}_F)\times\GL_k(\mathcal{O}_F)$ and of $\mathcal{G}_k\times\mathcal{G}_k$).
\begin{Ex} Assume that $\ell$ does not divide the order of $\mathcal{G}_n$. Then, in particular, 
$\ell$ does not divide $q^k -1$, and this implies (see III.2.9 of \cite{VigLivre}) that $\rho_0$
is projective. Along the lines of the proof of Theorem 4.3 in \cite{guiraud}, one can show 
that $\operatorname{i}_{\mathscr{P}}^G(\rho)$ is a pro-generator, 
where $\operatorname{i}_{\mathscr{P}}^G$ denotes the parahoric induction\footnote{For notations and definition, see
chapter 2.2 in \cite{guiraud}.} from the standard parahoric subgroup 
$\mathscr{P}\subset G$ uniquely characterized by $\mathscr{P}/\mathscr{P}(1) \cong \mathcal{G}_k\times \mathcal{G}_k$.
The important point is to use that $R[\mathcal{G}_n]-\textbf{Mod}$ is semisimple, so
$\operatorname{Hom}_{\mathcal{G}_n}(\operatorname{i}_{\mathcal{G}_k\times \mathcal{G}_k}^{\mathcal{G}_n}(\rho), X)$ is
non-zero for any subquotient (hence, subrepresentation) $X$ of $\operatorname{i}_{\mathcal{G}_k\times \mathcal{G}_k}^{\mathcal{G}_n}(\rho)$.
This implies
\begin{equation*}
\mathfrak{R}_R^{[M, \pi]_G}(G) \underset{\text{Morita}}{\cong} \mathscr{H}_R(G, \mathscr{P}, \rho),
\end{equation*}
where the object on the right hand side is the Hecke algebra of the type $(\mathscr{P}, \rho)$. By \cite{VigLivre}, Proposition III.3.6, 
one gets an isomorphism 
\begin{equation*}
\mathscr{H}_R(G, \mathscr{P}, \rho) \cong H_R(2, q^k) \cong \mathscr{H}_R(GL_2(F^k), \mathscr{I}),
\end{equation*}
where $F^k$ is the unramified extension of $F$ of degree $k$ and $\mathscr{I}$ is the Iwahori subgroup of
$GL_2(F^k)$. Hence, abbreviating $G' = GL_2(F^k)$ and denoting by $T'\subset G'$ the standard torus, we get
an equivalence of categories
\begin{equation}\label{equiv}
\mathfrak{R}_R^{[M, \pi]_G}(G) \cong \mathfrak{R}_R^{[T', 1]_{G'}}(G').
\end{equation}
The block on the right is the unipotent block of $G'$. It contains the trivial representation. 
\end{Ex}
The assumptions of the above example can be called banal for our context. 
Our aim in this article is to investigate the non-banal situation.
To this end, we will say in Section \ref{sec_progen} how one can write down a pro-generator $\mathscr{V}$ of the
block $\mathfrak{R}_R^{[M, \pi]_G}(G)$, depending only on 
the finite-group data $\mathcal{G}_k$ and $\rho_0$. In Section \ref{funcHecke} we will describe the multiplicative
structure of ``functional'' Hecke algebras, i. e. algebras of the form
\begin{equation*}
\mathscr{H}_R(G, \mathscr{P}, V) = \operatorname{End}_G(\operatorname{i}_{\mathscr{P}}^G(V))
\end{equation*}
where $V$ is a cuspidal (but not necessarily irreducible) representation of 
$\mathcal{M}=\mathcal{G}_k\times \mathcal{G}_k$. This will be used in Section \ref{tensor_dec} to
establish the main result of this article (Corollary \ref{MainResult}), which gives a 
decomposition of the Hecke algebra as a twisted\footnote{This construction is explained
in section \ref{twistedTensorprods}. We would also like to point the reader to the appendix
of \cite{cht}, where M.-F. Vign\'{e}ras gives a similar (but not identical) decomposition in a special case.} tensor product
\begin{equation*}
\mathscr{H}_R(G, \mathscr{P}, V) 
\cong R[\mathbb{Z}]\otimes_R^{\zeta} \mathscr{H}_R(\mathcal{G}_n, \mathcal{M} , V).
\end{equation*}
If we apply this to the constructed pro-generator $\mathscr{V}$ (which is of the
form $\operatorname{i}_{\mathscr{P}}^G(V)$ for a suitable $\mathcal{M}$-representation
$(\rho, V)$), this gives a strong relation between the block $\mathfrak{R}_R^{[M, \pi]_G}(G)$ 
and the block 
\begin{equation*}
\mathfrak{R}_R^{[\mathcal{M}, \rho]_{\mathcal{G}_n}}(\mathcal{G}_n)\subset \operatorname{Rep}_R(\mathcal{G}_n).
\end{equation*}
This second block is Morita equivalent to $\mathscr{H}_R(\mathcal{G}_n, \mathcal{M} , V)$ and consists of those
represenentations whose Jordan-H\"older quotients are isomorphic to Jordan-H\"older quotients
of the Harish-Chandra induced of $V$ from $\mathcal{M}$ to $\mathcal{G}_n$. 
This decomposition also suggests a way to generalize the equivalence (\ref{equiv}) (which we will
formulate as
Conjecture \ref{mainConj}): If we can establish an isomorphism of the ``finite'' Hecke algebras
\begin{equation}\label{isoOfFiniteHeckeAlgs}
\mathscr{H}_R(\mathcal{G}_n, \mathcal{M} , V) \cong
\mathscr{H}_R(\operatorname{GL}_2(q^k), \mathcal{B} , V_0)
\end{equation}
(where $\mathcal{B}$
denotes the Borel subgroup and $V_0$ denotes some pro-generator of the unipotent block
of $\operatorname{GL}_2(q^k)$) and if this isomorphism is compatible with the respective 
twisting parameters $\zeta$, then we get an isomorphism 
\begin{equation*}
\mathscr{H}_R(G, \mathscr{P}, V) \cong 
\mathscr{H}_R(\operatorname{GL}_2(F^m), \mathscr{I}, V_0)
\end{equation*}
and, consequently, an equivalence of $\mathfrak{R}_R^{[M, \pi]_{G}}(G)$ 
with the unipotent block of $\operatorname{GL}_2(F^k)$. The precise requirements on the isomorphism 
(\ref{isoOfFiniteHeckeAlgs}) are summarized in Conjecture \ref{BroueConj}.

It seems plausible that a similar decomposition of the Hecke algebra as a twisted tensor product can be 
established with some care in more general situations (i. e. for more complicated blocks). However, the trick used
in the proof of Theorem \ref{progenerator} does not carry over. So, even with a tensor decomposition theorem and
a suitable version of Conjecture \ref{BroueConj} at hand, the presented approach will not directly imply
an isomorphism of blocks \`a la (\ref{equiv}).
This is the reason why we decided to restrict our attention to blocks of the mentioned type.

\paragraph{Acknowledgements} The author wants to thank Gebhard B\"ockle for his encouragement
and his comments.

\section{Preliminaries on the pro-generator} \label{sec_progen}
Retain the notations of the previous chapter, in particular 
\begin{itemize}
\item $G = G_n =  \GL_n(F), \mathcal{G} = \mathcal{G}_n = \GL_n(q)$  with $n=2k$;
\item $\mathscr{P}\subset G$ the standard parahoric subgroup with reductive quotient
$\mathscr{P}/\mathscr{P}(1) \cong \mathcal{M} := \mathcal{G}_k\times \mathcal{G}_k$;
\item $\rho$ denotes an $\mathcal{M}$-representation of the form $\rho_0 \otimes \rho_0$, where 
$\rho_0$ is an irreducible supercuspidal $\mathcal{G}_k$-representation. 
\end{itemize}
Now, inside $R[\mathcal{M}]\!-\!\textbf{Mod}$ fix a projective cover $P$ of 
$\rho$. Define $V = P\oplus P^{\ast}$, where $P^{\ast}$ is the contragredient representation of $P$.
As usual, we do not distinguish between $V$ and its inflation along $\mathscr{P}(1)$
to a representation of $\mathscr{P}$.
\begin{Prop}
$V$ is cuspidal (in the sense that it vanishes under non-trivial 
Harish-Chandra functors).
\end{Prop}
\begin{proof}
The statement is certainly true for $P$ in place of $V$, as the projective 
cover is formed in the (supercuspidal) block associated to $\rho$, so all 
subquotients of $P$ are isomorphic to $\rho$, 
i. e. are cuspidal. That this property is not disturbed by adding
the contragredient follows from \cite{BDC}, 2.2e Lemma.
\end{proof}
Let $[M,\pi]$ be the supercuspidal pair associated to the type $(\mathscr{P},\rho)$ (cf. \cite{guiraud}, 2.2).
\begin{Thm}\label{progenerator}
The parahorically induced $\mathscr{V} = \operatorname{ind}_{\mathscr{P}}^G(V)$ is a pro-generator of the block
$\mathfrak{R}^{[M, \pi]_G}(G)$.
\end{Thm} 
\begin{proof}
It is clear that $\mathscr{V}$ is finitely generated and projective, as these properties are respected by
parahoric induction (and as by \cite{VigLivre}, the projective cover is finite-length, hence finitely generated).
Now, by Frobenius reciprocity and transitivity of parahoric induction, we have
\begin{equation*}
\operatorname{Hom}_G(\mathscr{V}, W) = 
\operatorname{Hom}_{\mathcal{G}}(\textit{i}_{\mathcal{M}}^{\mathcal{G}}(V), W^{\mathscr{K}})
\end{equation*}
for any non-zero $W\in \mathfrak{R}^{[M, \pi]}(G)$, 
where $\mathscr{K}=\GL_n(\mathcal{O}_F)$ denotes the maximal compact subgroup 
and we can assume that $W$ is irreducible.
As $W$ is level-$0$, we have $W^{\mathscr{K}} \neq 0$. Thus, in order to show that the above Hom-set
does not vanish, it is sufficient (by \cite{guiraud}, 2.16) to show that there exists a non-zero map 
\begin{equation*}
\textit{i}_{\mathcal{M}}^{\mathcal{G}}(V) \rightarrow X
\end{equation*}
for any subquotient $X$ of $\textit{i}_{\mathcal{M}}^{\mathcal{G}}(\rho)$.
By (\cite{BDC}, 2.4a), this representation has precisely two subquotients, each one occurring with
multiplicity $1$. So we can write 
\begin{equation*}
0 \rightarrow Y \rightarrow 
\textit{i}_{\mathcal{M}}^{\mathcal{G}}(\rho) \rightarrow Z 
\rightarrow 0.
\end{equation*}
Now, as $P\twoheadrightarrow \rho$ and Harish-Chandra induction is exact and respects direct sums, we get a surjection
$\textit{i}_{\mathcal{M}}^{\mathcal{G}}(V) \twoheadrightarrow \rho$. Composing this with the 
projection from the short exact sequence, we get a non-zero map if $X=Z$.\\
For $X=Y$, we use that taking the contragredient is an exact contravariant functor, so there is a projection
$\textit{i}_{\mathcal{M}}^{\mathcal{G}}((\rho)^{\ast})\cong  \textit{i}_{\mathcal{M}}^{\mathcal{G}}(\rho)^{\ast} 
\twoheadrightarrow Y$ (the isomorphism is \cite{BDC}, 2.2e; note that we have used here that an irreducible
representation is isomorphic to its contragredient, as mentioned in the remarks before Lemma 2.2e in 
\cite{BDC}).
By \cite{VigLivre}, III.2.9 Theoreme 2.b, all irreducible constituents of $P$ are of
the form $\rho$, i. e. $\rho \hookrightarrow P$. But this
implies the existence of a surjection $P^{\ast}\twoheadrightarrow (\rho)^{\ast}$. This can again be induced and composed
with the map above to finish off the case $X=Y$. 
\cite{BDC}.)
\end{proof}

\section{Functional Hecke-algebras}\label{funcHecke}
We use the notation from above, except that we allow
$(V,\sigma)$ to be any finite-dimensional
$\mathcal{M}$-representation which we assume to be cuspidal (in the sense that all 
non-trivial Harish-Chandra functors vanish on it) but not necessarily irreducible. 
We will not distinguish between $V$ and its inflation along $\mathscr{P}(1)$ to a representation
of $\mathscr{P}$. Our aim is to study the $R$-algebra
\begin{equation*}
\mathscr{H}_R(G, \mathscr{P}, V) = 
\operatorname{End}_G \left( \operatorname{ind}_{\mathscr{P}}^G(V) \right).
\end{equation*}
Following \cite{VigLivre}, I.8.5, this algebra can be characterized as the convolution algebra of functions
\begin{equation*}
\varphi: G \rightarrow \operatorname{End}_R(V)
\end{equation*}
supported on a finite number of double cosets $\mathscr{P}g\mathscr{P}$ such that 
\begin{equation}\label{HeckeChar}
\varphi(p_1gp_2) = \sigma(p_1)\varphi(g)\sigma(p_2)\qquad \text{ for all }g\in G, p_1, p_2 \in \mathscr{P}.
\end{equation}
If $\mathscr{H}_R(\mathscr{P}g\mathscr{P}, \mathscr{P}, V)\subset \mathscr{H}_R(G, \mathscr{P}, V)$ denotes the subspace of functions
supported on the single coset $\mathscr{P}g\mathscr{P}$, we have a vector space decomposition
\begin{equation*}
\mathscr{H}_R(G, \mathscr{P}, V) = \bigoplus_{\eta\in \tilde{\mathscr{W}}}\mathscr{H}_R(\mathscr{P}\eta \mathscr{P}, \mathscr{P}, V),
\end{equation*}
where $\tilde{\mathscr{W}}$ denotes a set of representatives for $\mathscr{P}\backslash G /\mathscr{P}$. With $\varpi = \varpi_F$, let
\begin{equation*}
\mathscr{W} = \left\{\text{ monomial $2k\times 2k$-matrices with entries in $\varpi^{\mathbb{Z}}$}\right\}
\end{equation*}
be the affine Weyl group of $G$. Then, according to e. g. \cite{vigneras2003schur}, p. 43, we can 
take as $\tilde{\mathscr{W}}$ a set of representatives of $Y\backslash \mathscr{W}/Y$, where
$Y\cong S_k\times S_k$ denotes the Young subgroup associated to the partition $(k,k)$ of $n=2k$.
This means we can take
\begin{equation*}
\tilde{\mathscr{W}} = \left\{ v.\operatorname{diag}(\varpi^{a_1}, \ldots, \varpi^{a_{2k}})\;\bigl|\;
a_i\in \mathbb{Z}, v\in \mathbb{W}\right\},
\end{equation*}
where $\mathbb{W}$ is a fixed set of representatives for $Y\backslash S_{n}/Y$, which we assume for 
convenience to include $1 = \operatorname{diag}(1, \ldots, 1)$ and $w = (1, k+1)(2, k+2) \ldots (k, 2k)$, or,
in matrix notation,
\begin{equation*}
w = \left( \begin{matrix}
& \begin{smallmatrix} 1 &&\\ & \ddots\\ &&1\end{smallmatrix}\\
\begin{smallmatrix} 1 &&\\ & \ddots\\ &&1\end{smallmatrix}
\end{matrix}\right).
\end{equation*} 
Denote
\begin{equation*}
W = \left\{v.\operatorname{diag}(\varpi^{a_1}, \ldots, \varpi^{a_{2k}})\in \tilde{\mathscr{W}}
\;\bigl|\; a_1 = a_2 = \ldots = a_k, a_{k+1} = a_{k+2} = \ldots = a_{2k}, v\in \{1, w\}\right\}
\end{equation*}
and remark that this forms a subgroup of $\mathscr{W}$ isomorphic to the affine Weyl group of 
$\GL_2(F)$.
\begin{Lem}
The support of $\mathscr{H}_R(G, \mathscr{P}, V)$ is contained in $\mathscr{P}W\mathscr{P}$, i. e. 
\begin{equation*}
\mathscr{H}_R(\mathscr{P}\eta \mathscr{P}, \mathscr{P}, V) \neq 0 \text{ and } \eta\in \tilde{\mathscr{W}} \Rightarrow
\eta \in W.
\end{equation*}
\end{Lem}
\begin{proof}
This is Lemma 1.1 (p. 15) in \cite{Howe}. (The representation there is irreducible, but the argument
only needs cuspidality.)
\end{proof}
Now, if $\eta = v.\operatorname{diag}(\varpi^{a_1}, \ldots, \varpi^{a_{2k}})$, write 
$\delta = a_1 - a_{k+1}$. Moreover, abbreviate $\mathcal{O} = \mathcal{O}_F$ and 
$\mathfrak{P} = \varpi.\mathcal{O}$.
\begin{Lem} Denote $\mathscr{P}^{(\eta)} = \mathscr{P}\cap \eta \mathscr{P}\eta^{-1}$, then we have
\begin{itemize}
\item \begin{equation*}
\mathscr{P}^{(\eta)} = \left(\begin{smallmatrix}
\GL_k(\mathcal{O}) & \mathbb{M}_{k\times k}(\varpi^{\delta}\mathcal{O})\\
\mathbb{M}_{k\times k}(\mathfrak{P}) & \GL_k(\mathcal{O})
\end{smallmatrix}\right) \qquad \text{if $v = 1$ and $\delta \geq 0$};
\end{equation*}
\item \begin{equation*}
\mathscr{P}^{(\eta)} = \left(\begin{smallmatrix}
\GL_k(\mathcal{O}) & \mathbb{M}_{k\times k}(\mathcal{O})\\
\mathbb{M}_{k\times k}(\varpi^{1-\delta}\mathfrak{P}) & \GL_k(\mathcal{O})
\end{smallmatrix}\right) \qquad \text{if $v = 1$ and $\delta < 0$};
\end{equation*}
\item \begin{equation*}
\mathscr{P}^{(\eta)} = \left(\begin{smallmatrix}
\GL_k(\mathcal{O}) & \mathbb{M}_{k\times k}(\varpi^{1-\delta}\mathcal{O})\\
\mathbb{M}_{k\times k}(\mathfrak{P}) & \GL_k(\mathcal{O})
\end{smallmatrix}\right) \qquad \text{if $v = w$ and $\delta \leq 0$};
\end{equation*}
\item \begin{equation*}
\mathscr{P}^{(\eta)} = \left(\begin{smallmatrix}
\GL_k(\mathcal{O}) & \mathbb{M}_{k\times k}(\mathcal{O})\\
\mathbb{M}_{k\times k}(\mathfrak{P}^{\delta}) & \GL_k(\mathcal{O})
\end{smallmatrix}\right) \qquad \text{if $v = w$ and $\delta > 0$}.
\end{equation*}
\end{itemize}
\end{Lem}
\begin{proof}
Straight-forward matrix calculation.
\end{proof}
We easily get an identification
\begin{equation*}
\mathscr{H}_R(\mathscr{P}\eta \mathscr{P}, \mathscr{P}, V) = \left\{ \psi: V\rightarrow V \;\Bigl|\; \psi \circ\sigma(x) = 
\sigma(\eta x \eta^{-1})\circ \psi \text{ for all }x\in \mathscr{P}^{(\eta^{-1})}\right\}
\end{equation*}
\begin{equation*}
\cong
\begin{cases}
I_1 =\operatorname{End}_{\mathcal{M}}(V) & \text{ if $\eta$ diagonal;}\\
I_w = \operatorname{Hom}_{\mathcal{M}}(V^w, V) \cong \operatorname{Hom}_{\mathcal{M}}(V, V^w) & \text{ otherwise;}
\end{cases}
\end{equation*}
where $V^w$ denotes the $w$-conjugate representation of $V$, i. e. we let $m\in\mathcal{M}$ act on $v\in V$ as $(wmw^{-1}).v$.
We will consequently use the notation $[\eta]_f$ for the element of $\mathscr{H}_R(P\eta P, P, V)$
with $\varphi(\eta) = \psi = f$, where $f$ is an element of $I_1$ or $I_w$ depending on whether $\eta$ is
diagonal or not.\\
Fix the following notation, which is clearly motivated by the isomorphism of $W$ with the affine Weyl group of $\GL_2(F)$:
\begin{itemize}
\item $t = \left(\begin{matrix} & \begin{smallmatrix}
1 &&\\
& \ddots \\
&&1
\end{smallmatrix}
\\
\begin{smallmatrix}
\varpi &&\\
& \ddots \\
&&\varpi
\end{smallmatrix}\end{matrix}\right)$;
\item $w' = twt^{-1}$;
\item $W_0 = \{x \in W \;|\; \operatorname{det}(x) = \pm 1\}$.
\end{itemize}
Using this isomorphism, $W$ is an affine Coxeter group with length function $l$
satisfying $l(w) = l(w') = 1, l(t)=0$. On the other hand, $\mathscr{W}$ is also
an affine Coxeter group equiped with a length function $l'$. 
\begin{Prop}
$l'|W$ is a constant multiple of $l$.
\end{Prop}
\begin{proof}
This is (5.6.14) (p. 192) of \cite{bushnell1993admissible}.
Notice that there it is written that $l'|W = n.l$, but this is not correct and it should be $l'|W = n^2.l$.
\end{proof}
Of fundamental importance is the operator
\begin{equation*}
T^{\ast} = \sum_{g\in \GL_k(q)} \left( \begin{smallmatrix} g & \\
 & -g^{-1} \end{smallmatrix}\right) \in R[\mathcal{M}] \subset R[\mathcal{G}],
\end{equation*}
which already appears in \cite{Howe}, p. 10, (5.13). $T^{\ast}$ lies in $I_w$ but generally not in 
$I_1$, as a quick calculation reveals.
\begin{Prop}
\begin{enumerate}[label={\upshape(\roman*)}]
\item If $\ell$ divides $q-1$, then $(T^{\ast})^2 = 0$.
\item If $f$ is in $I_1$ (resp. $I_w$), then $T^{\ast}f = fT^{\ast}$ in $I_w$ (resp. $I_1$).
\end{enumerate}
\end{Prop}
\begin{proof}
(ii) is obvious, so we only treat (i). First, we have
\begin{equation*}
(T^{\ast})^2 = 
\sum_{g\in \GL_k(q)} \sum_{h\in \GL_k(q)} \left( \begin{smallmatrix}
gh \\ & g^{-1} h^{-1}\end{smallmatrix}\right)
= 
\sum_{g\in \GL_k(q)} \sum_{h\in \GL_k(q)} \left( \begin{smallmatrix}
g \\ & hg^{-1} h^{-1}\end{smallmatrix}\right)
= 
\sum_{g\in \GL_k(q)} \sum_{h\in C(g)} (\# Z_h)\left( \begin{smallmatrix}
g \\ & h\end{smallmatrix}\right),
\end{equation*}
where $C(g) = \{ hgh^{-1} \;|\; h\in \GL_k(q)\}$ and 
$Z_h = \{x\in \GL_k(q) \;|\; xhx^{-1} = h\}$.
As $k_F^{\times}.\mathbf{1}$ is a subgroup of
$Z_h$ for any $h\in G$, $\ell$ divides all the coefficients
$\#Z_h$ in the above expression.
\end{proof}
It is easy to see that the first condition in the Proposition is not true in general. In fact, one can produce 
examples (i. e. list combinations of $\ell, q, n$) where $T^{\ast}$ is not even nilpotent.
\begin{Def}
\begin{enumerate}[label={\upshape(\roman*)}]
\item Denote by $\tau = \tau_{\ell,q, k} := q^{(k^2)} \operatorname{mod}\ell$. It follows by our assumption $(\ell, q) = 1$ that
$\tau$ is non-zero, so $\tau\in\{1,\ldots, \ell -1\}$.
\item If $f$ is in $I_1$ or $I_w$, $\eta \in W$ and $j\in \mathbb{N}$, we write $[\eta]_f^j$ for $[\eta]_{(T^{\ast})^jf}$
and $[\eta]^j$ for $[\eta]^j_1$. Observe that this makes sense only
for even $j$ (in the case that $f\in I_1$ and $\eta$ diagonal or $f\in I_w$ and $\eta$ non-diagonal) or
for odd $j$ (in the case that $f\in I_w$ and $\eta$ diagonal or $f\in I_1$ and $\eta$ non-diagonal).
Clearly, $[\eta]_f = [\eta]_f^0$.
\end{enumerate}
\end{Def}

Our aim is to prove the following:
\begin{Thm}\label{weiobg084ht802hrth3rieh98z25zt934hg98rghwibte32}
Let $[\eta]_f, [\delta]_g\in \mathscr{H}_R(G, \mathscr{P}, V)$ and, in the case $l(\eta) >1$, write
$\eta = \eta' v$ with $v\in S = \{w,w'\}$. Then
\begin{equation*}
[\eta]_f\ast [\delta]_g = \begin{cases}
[\eta\delta]_{fg} & \text{ if }l(\eta\delta) = l(\eta)+l(\delta);\\
\tau.[\eta\delta]_{fg} + [\eta'\delta]_{fg}^1 & \text{ if }l(\eta) = l(\delta) =1 \text{ and } l(\eta\delta) = 0.
\end{cases}
\end{equation*}
(The theorem does not make a statement for the case $l(\eta\delta) \neq l(\eta)+l(\delta)$ and
$\operatorname{max}(l(\eta), l(\delta)) >1$.)
\end{Thm}
\begin{Cor}
The above theorem completely determines the multiplication in $\mathscr{H}_R(G, \mathscr{P}, V)$.
\end{Cor}
\begin{proof}[Proof of the corollary.]
We show how the theorem can be used inductively to evaluate the product $[\eta]_f\ast[\delta]_g$, 
where the induction is on the number $l_{\eta,\delta} = l(\eta) + l(\delta)$. The above theorem 
allows evaluation if $l_{\eta, \delta} \leq 2$. Moreover, the above theorem establishes
\begin{equation*}
\text{Fact A}:\qquad l(\eta \delta)<l_{\eta,\delta} \Rightarrow [\eta]_f\ast[\delta]_g = \sum_I
[\epsilon_i]_{h_i} \text{ with } l(\epsilon_i)<l_{\eta,\delta} \text{ for all } i\in I
\end{equation*}
under the condition that $l_{\eta, \delta} \leq 2$.\\
For the induction step, assume that Fact A and multiplication of $[\eta]_f$ and $[\delta]_g$
are established as long as $l_{\eta, \delta}<m$ for some $m$.
We explain how to prove Fact A and how to evaluate $[\eta]_f\ast[\delta]_g$ if 
$l_{\eta,\delta} = m$: \\
If $l(\eta\delta) = m$, we can directly apply the above theorem, so assume
inequality.
Moreover, assume for the moment that $l(\eta)\geq 2$. This means, we can break up $\eta = \eta_1 \eta_2$ in a way that 
$l(\eta_1)+l(\eta_2) = l_{\eta_1,\eta_2}$, both $l(\eta_1)$ and $l(\eta_2)$ are nonzero,
and $\eta_1$ is diagonal. Explicitly, if $\eta = t^{\alpha}w_1\ldots w_a$ is a minimal expression, we can take
\begin{equation*}
(\eta_1) (\eta_2) = 
\begin{cases}
(t^{\alpha}w_1)(w_2\ldots w_a) & \text{ if $\alpha$ is odd};\\
(t^{\alpha - 1}w_1')(tw_2\ldots w_a) & \text{ if $\alpha>0$ is even};\\
(t^{-1}w_1')(t w_2\ldots w_a) & \text{ if $\alpha=0, w_1 = w$};\\
(tw_1')(t^{-1} w_2\ldots w_a)  & \text{ if $\alpha=0, w_1 = w'$}.
\end{cases}
\end{equation*}
Here, the operator $w_1\mapsto w_1'$ is given by interchanging the symbols ``$w$'' and ``$w'$''.

 In this situation, the inequality 
$l(\eta\delta)<m$ implies $l(\eta_2\delta) < l_{\eta_2, \delta} = l_{\eta,\delta} - l(\eta_1) < m$ (as can
be easily seen from writing $\delta = w_{a+1}\ldots w_{a+b}t^{\beta}$ and noting $
l(\eta\delta)<m \Leftrightarrow w_a = w_{a+1}$).
So, in 
\begin{equation*}
[\eta]_f\ast [\delta]_g = [\eta_1]_{1}\ast([\eta_2]_f\ast[\delta]_g),
\end{equation*}
the expression in the brackets can be calculated because
$l_{\eta_2, \delta}<m$:
\begin{equation*}
[\eta]_f\ast [\delta]_g =
[\eta_1]_1 \ast \sum_I
[\epsilon_i]_{h_i},
\end{equation*}
where $l(\epsilon_i) <l_{\eta_2, \delta} = m - l(\eta_1)$ (because of Fact A applied to 
$[\eta_2]_f\ast[\delta]_g$). Therefore, each summand 
$[\eta_1]_1 \ast [\epsilon_i]_{h_i}$ can be calculated. Moreover,
if 
\begin{equation*}
[\eta_1]_1 \ast [\epsilon_i]_{h_i} = \sum_{j\in J_i} [\nu_j^i]_{h_j^i},
\end{equation*}
we can use the inequality $l_{\eta_1, \epsilon_i} < m - l(\eta_1) + l(\eta_1) = m$
to apply either Fact A (if $l_{\eta_1, \epsilon_i} < l(\eta_1)+l(\epsilon_i)$) or
the above theorem (if $l_{\eta_1, \epsilon_i} = l(\eta_1)+l(\epsilon_i)$), 
giving $l(\nu_j^i) < m$ and establishing Fact A for $[\eta]_f\ast[\delta]_g$.

If the initial assumption on $\eta$ is not met, we can use an analogous argument from the right
(i. e. breaking up $\delta = \delta_2 \delta_1$).
\end{proof}
\subsection{Proof of Theorem \ref{weiobg084ht802hrth3rieh98z25zt934hg98rghwibte32}}
Our main tool is the following formula:
\begin{Prop}
Let $[\eta]_f, [\delta]_g \in \mathscr{H}_R(G, \mathscr{P}, V)$. 
Then
\begin{equation*}
[\eta]_f \ast [\delta]_g = \sum_{\epsilon\in W\cap \mathscr{P}\eta \mathscr{P} \delta \mathscr{P}} [\epsilon]_{h_{\epsilon}}, 
\end{equation*}
where
\begin{equation}\label{VignerasFormula}
h_{\epsilon} = \sum_{\substack{ (k_1, k_2)\in \left(\mathscr{P}^{(\eta^{-1})}\backslash \mathscr{P}\right)\times \left(\mathscr{P}^{(\delta^{-1})}\backslash \mathscr{P}\right)\\ \text{ such that}\;
k_0 := \epsilon k_2^{-1} \delta^{-1}k_1^{-1} \eta^{-1}\in \mathscr{P}}} \rho(k_0) \circ f \circ \rho(k_1) \circ g \circ \rho(k_2).
\end{equation}
\end{Prop}
For notational simplicity, we will refer to the pairs $(k_1, k_2)$ which fulfill the condition described under the summation sign
as ``admissible'' pairs (with respect to the problem of evaluating $h_{\epsilon}$).
\begin{proof} This is a slightly reformulated version of the formula given in \cite{vigneras1998induced}, II.2.
\end{proof}
It is important to note that the indices under both summation signs do not depend on $R$, $V$,$f$ and $g$. This means, we can get a 
lot of information just from looking at the spherical integral Hecke algebra $\mathscr{H}_{\mathbb{Z}}(G, \mathscr{P}, 1)$.
\begin{Prop}
$\# (\mathscr{P}x\mathscr{P}/\mathscr{P}) = q^{l'(x)}$ for $x\in W$.
\end{Prop}
\begin{proof}
By \cite{bushnell1993admissible}, Chapter 5.5, there is a subgroup $\mathfrak{M}^{\times}\subset \mathscr{P}$ such that 
$\mathscr{P} = \mathfrak{M}^{\times}\mathscr{I}$ (with $\mathscr{I}$ denoting the Iwahori subgroup) 
and $W$ normalizes $\mathfrak{M}^{\times}$.
(In our case, $\mathfrak{M}^{\times}$ is $\GL_k(\mathcal{O})\times \GL_k(\mathcal{O})$ understood
as block matrices.)
So we can use the solid bijections in
\begin{equation*}
\xymatrix{
\mathscr{P}x\mathscr{P}/\mathscr{P} \ar@{=}[r] & \mathscr{I}x\mathscr{P}/\mathscr{P} \ar@{<.>}[r]\ar@{<->}[d] & \mathscr{I}x\mathscr{I}/\mathscr{I}\ar@{<->}[d]\\
& \mathscr{I}/(\mathscr{I}\cap x\mathscr{P}x^{-1}) \ar@{<->}[r]& \mathscr{I}/(\mathscr{I}\cap x\mathscr{I}x^{-1})
}
\end{equation*}
to establish that the dotted arrow is a bijection.
The nature of the vertical arrows is clear. The bottom horizontal arrow is a bijection
by use of the identity
\begin{equation*}
\mathscr{I}\cap x\mathscr{P}x^{-1} = \mathscr{I}\cap x\mathscr{I}x^{-1},
\end{equation*}
which follows from a straightforward matrix computation, invoking the block 
structure of $x$ (i. e. the fact that $x\in W$).
Finally, the cardinality of $\mathscr{I}x\mathscr{I}/\mathscr{I}$ is 
$q^{l'(x)}$ e. g. by \cite{bushnell1993admissible},
5.4.3.(ii).
\end{proof}
\begin{Thm}\label{Theorem310}
If $l(\eta\delta) = l(\eta)+l(\delta)$ and
$\eta, \delta \in W$, then
$[\eta]\ast[\delta] = [\eta\delta]$ in $\mathscr{H}_{\mathbb{Z}}(G, \mathscr{P},1)$.
\end{Thm}
\begin{proof}
Using that $l$ and $l'$ differ by a non-zero factor, we can use Lemma 5.6.12 of \cite{bushnell1993admissible} 
to deduce from the addition of lengths that $\mathscr{P}\eta \mathscr{P} \delta \mathscr{P}  = \mathscr{P}\eta \delta \mathscr{P}$. This implies
that
$[\eta][\delta] = \lambda[\eta\delta]$ for some $\lambda\in \mathbb{Z}$. Now,
for a fixed complex-valued Haar measure $d$ on $G$, one defines the augmentation map
\begin{equation*}
\epsilon: \mathscr{H}_{\mathbb{Z}}(G, \mathscr{P},1)\rightarrow \mathbb{C} \qquad  f\mapsto \int_G f(t) dt
\end{equation*}
and checks that it respects the multiplication. (This follows from Proposition 4.2 of \cite{BH},
if one takes for $V$ the trivial representation and if one uses the embedding
$\mathscr{H}_{\mathbb{Z}}(G, \mathscr{P},1)\subset \mathscr{H}_{\mathbb{C}}(G, \mathscr{P},1)$.) Therefore,
\begin{equation*}
q^{l'(\eta)}q^{l'(\delta)} = \lambda
q^{l'(\eta\delta)}.
\end{equation*}
This proves $\lambda = 1$.
\end{proof}
This is sufficient to deduce the first part of the theorem:
\begin{Cor} Let $[\eta]_f, [\delta]_g \in \mathscr{H}(G, \mathscr{P}, V)$ with $\epsilon, \delta\in 
W$ and such that $l(\eta\delta) = l(\eta)+ l(\delta)$.
Then
\begin{equation*}
[\eta]_f\ast [\delta]_g = [\eta\delta]_{fg}.	
\end{equation*}
\end{Cor}
\begin{proof}
Theorem \ref{Theorem310} implies that there occurs only one admissible pair $(k_1, k_2)$ when
using formula (\ref{VignerasFormula}) to calculate $[\eta]\ast[\delta]$ in $\mathscr{H}_{\mathbb{Z}}(G, \mathscr{P}, 1)$.
It is obvious that the choice $k_1 = 1, k_2 = 1$ is admissible, hence this is the unique admissible
pair. By the remark below the formula, this fact is still valid when  
calculating $[\eta]_f\ast [\delta]_g$, hence we get $h_{\eta\delta} = f\circ g$.
\end{proof}
In order to establish the second part -- as $t^a$ is universally length-adding and central if $2|a$ --
it suffices to evaluate
\begin{equation*}
[\eta]_f\ast[\delta]_g
\end{equation*}
(where we will suppress the symbol $\ast$ from now on) for the 8 choices for $\eta, \delta \in \{w, w', tw = w't, t^{-1}w' = wt^{-1}\}$ where there
is no addition of lengths.
We first calculate
\begin{itemize}
\item $\mathscr{P}^{(w)} = \mathscr{P}^{(wt^{-1})} = 
\left(\begin{smallmatrix}
\GL_k(\mathcal{O}) & \mathbb{M}_{k\times k}(\mathfrak{P})\\
\mathbb{M}_{k\times k}(\mathfrak{P}) & \GL_k(\mathcal{O})
\end{smallmatrix}\right)$;
\item $\mathscr{P}^{(w')} = \mathscr{P}^{(w't)} = 
\left(\begin{smallmatrix}
\GL_k(\mathcal{O}) & \mathbb{M}_{k\times k}(\mathcal{O})\\
\mathbb{M}_{k\times k}(\mathfrak{P}^2) & \GL_k(\mathcal{O}).
\end{smallmatrix}\right)$.
\end{itemize}
Therefore, in the first two cases we can take $\{\alpha_x = \left(\begin{smallmatrix}1 & x \\ & 1\end{smallmatrix}\right)\;|\;
x\in\mathbb{M}_{k\times k}(\mathbb{F}_q) \}$ and in the last two cases we can take
$\{\beta_y = \left(\begin{smallmatrix}1 &  \\ \varpi y & 1\end{smallmatrix}\right)\;|\;
y\in\mathbb{M}_{k\times k}(\mathbb{F}_q) \}$ as set of representatives for $\mathscr{P}^{(\eta)}\backslash \mathscr{P}$.
\begin{Prop}
\begin{equation*}
\mathscr{P}w\mathscr{P}w\mathscr{P} \subset \left(\mathscr{P} \cup \mathscr{P}w\mathscr{P} \cup \bigcup_{x\in X} \mathscr{P}x\mathscr{P}\right),
\end{equation*}
where $X$ can be taken as a subset of $\tilde{\mathscr{W}} - W$.
\end{Prop}
\begin{proof}
This follows easily from the observation
\begin{equation*}
\mathscr{P}w\mathscr{P}w\mathscr{P} \subset \GL_{n}(\mathcal{O}) = \bigcup_{x\in S_{n}} \mathscr{I}x\mathscr{I} = \bigcup_{x\in \tilde{\mathscr{W}}\cap S_{n}} \mathscr{P}x\mathscr{P}
\end{equation*}
and the decompositon
\begin{equation*}
\tilde{\mathscr{W}}\cap S_{n} = \bigl((\tilde{\mathscr{W}} - W)\cap S_{n}\bigr) \sqcup (W \cap S_{n}) =
\bigl((\tilde{\mathscr{W}} - W)\cap S_{n}\bigr) \sqcup \{1, w\}.
\end{equation*}
\end{proof}
\begin{Thm}[Case 1: $\eta = \delta = w$]
In $\mathscr{H}_R(G, \mathscr{P}, V)$, we have the identity
\begin{equation*}
[w]_f[w]_g = \tau[1]_{fg} + [w]_{fg}^1.
\end{equation*}
\end{Thm}
\begin{proof}
By the above proposition, we know that $[w]_f [w]_g = [1]_{X} + [w]_Y$, so we just have to work out $X$ and $Y$.
For the first summand, we first remark that by \cite{VigLivre}, I.3.4.(iv), in $\mathscr{H}_{\mathbb{Z}}(G, P, 1)$, 
we have $[w][w] = \lambda[1] + \ldots$, where $\lambda = \# \mathscr{P}w\mathscr{P}/\mathscr{P} = \#\mathscr{P}/\mathscr{P}^{(w)} = \#\mathscr{P}^{(w)}\backslash \mathscr{P} = q^{(k^2)}$.
This means, we have to locate $q^{(k^2)}$ admissible pairs in formula (\ref{VignerasFormula}). This is easily achieved by fixing
$k_1 = 1$ and letting $k_2$ run through all $q^{(k^2)}$ possible values in $\mathscr{P}^{(w)}\backslash \mathscr{P}$. Applied to 
$\mathscr{H}_R(G, \mathscr{P}, V)$, this means we get $X = q^{(k^2)} f\circ g = \tau fg$.\\
For the second factor, the condition of the formula is $k_0= wk_{x_1}^{-1}wk_{x_2}^{-1}w \in \mathscr{P}$, which 
boils down to
\begin{equation*}
\left( \begin{smallmatrix} x_2 &\textbf{1} \\ \textbf{1}+x_1x_2 & x_1\end{smallmatrix} \right) \in \mathscr{P}.
\end{equation*}
This is possible if and only if $x_1, x_2$ are invertible and fulfill $x_1x_2 = -\textbf{1}$.
This gives $\#\operatorname{GL}_k(q)$-many admissible pairs and rise to 
\begin{equation*}
Y = \sum_{z\in \operatorname{GL}_k(q)} \rho\left(\begin{smallmatrix} z &\textbf{1} \\ 
 & -z^{-1}\end{smallmatrix} \right)\circ f \circ g = T^{\ast}fg.
\end{equation*}
\end{proof}
\begin{Prop}
\begin{equation*}
\mathscr{P}tw\mathscr{P}wt^{-1}\mathscr{P} \subset \mathscr{P} \cup \mathscr{P}w'\mathscr{P} \cup \bigcup_{x\in Y} \mathscr{P}x\mathscr{P},
\end{equation*} 
where $Y$ can be taken as a subset of $\tilde{\mathscr{W}} - W$.
\end{Prop}
\begin{proof}
We use that $l(tw) = l(t)+l(w)$ to write
\begin{equation*}
\mathscr{P}tw\mathscr{P}wt^{-1}\mathscr{P} = \mathscr{P}t\mathscr{P}w\mathscr{P}w\mathscr{P}t^{-1}\mathscr{P} \subset \mathscr{P}t\mathscr{P}t^{-1}\mathscr{P} \cup \mathscr{P}t\mathscr{P}w\mathscr{P}t^{-1}\mathscr{P} \cup \bigcup_{x\in X} \mathscr{P}t\mathscr{P}x\mathscr{P}t^{-1}\mathscr{P}
\end{equation*}
\begin{equation*}
= \mathscr{P} \cup \mathscr{P}w'\mathscr{P} \cup \bigcup_{x\in X} \mathscr{P}txt^{-1}\mathscr{P}.
\end{equation*}
It is clear that conjugation by $t$ preserves the property of being a member of $W$ or not.
\end{proof}
\begin{Thm}
In $\mathscr{H}_R(G, \mathscr{P}, V)$ we have the identity
\begin{equation*}
[tw]_1[wt^{-1}]_1 = \tau [1]_1 + [w']^1.
\end{equation*}
\end{Thm}
\begin{proof}
By the proposition and the same reasoning as above, we see that
$[tw]_1 [wt^{-1}]_1 = [1]_X + [w']_Y$, and that we have to 
locate $q^{(k^2)}$ admissible pairs in formula (\ref{VignerasFormula})
in order to calculate $X$, which can again achieved by fixing
$k_1 = 1$, yielding $X = q^{(k^2)}$.
For the second summand, we can take $k_1 = \alpha_x$
and $k_2 = \beta_y$. We have
to work out for which values of $x,y$ we have
\begin{equation*}
twt^{-1}k_2^{-1}twk_1^{-1}wt^{-1} \in \mathscr{P}.
\end{equation*}
This boils down to
\begin{equation*}
\left(\begin{smallmatrix} y & \varpi^{-1}(\textbf{1}+xy)\\ \varpi & x \end{smallmatrix}\right) \in \mathscr{P}.
\end{equation*}
This is clearly fulfilled if and only if $x, y\in \operatorname{GL}_k(q)$ and fulfill
$xy = -1$. This gives the second summand as $[twt^{-1}]_{T^{\ast}} = [w']^1$.
\end{proof}
\begin{Cor}[Case 2: $\eta = tw, \delta = tw'$]
\begin{equation*}
[tw]_f[wt^{-1}]_g = \tau[1]_{fg} + [twt^{-1}]_{fg}^1.
\end{equation*}
\end{Cor}
\begin{proof}
Write
\begin{equation*}
[tw]_f[wt^{-1}]_g = [1]_f [tw]_1 [wt^{-1}]_1 [1]_g
\end{equation*}
and apply the theorem above.
\end{proof}
\begin{Cor}[Case 3: $\eta = \delta = w'$]
\begin{equation*}
[w']_f[w']_g = \tau[1]_{fg} + [w']_{fg}^1.
\end{equation*}
\end{Cor}
\begin{proof}
Write
\begin{equation*}
[w']_f[w']_g = 
[tw]_1[t^{-1}]_f[t]_g[wt^{-1}]_1 = 
[tw]_1[1]_{fg}[wt^{-1}]_1 
= [tw]_1 [wt^{-1}]_1 [1]_{fg}
\end{equation*}
and apply the theorem above.
\end{proof}
\begin{Prop}
\begin{equation*}
\mathscr{P}wt\mathscr{P}t^{-1}w\mathscr{P} \subset \mathscr{P} \cup \mathscr{P}w\mathscr{P} \cup \bigcup_{x\in Y} \mathscr{P}x\mathscr{P},
\end{equation*} 
where $Y$ can be taken as a subset of $\tilde{\mathscr{W}} - W$.
\end{Prop}
\begin{proof}
This becomes clear when writing
\begin{equation*}
\mathscr{P}wt\mathscr{P}t^{-1}w\mathscr{P} = \mathscr{P}w\mathscr{P}t\mathscr{P}t^{-1}\mathscr{P}w\mathscr{P} = \mathscr{P}w\mathscr{P}w\mathscr{P}.
\end{equation*}
\end{proof}

\begin{Thm}
\begin{equation*}
[wt^{-1}]_1 [tw]_1 = \tau[1]_1 + [w]^1.
\end{equation*}
\end{Thm}
\begin{proof}
Analogous to the theorem above.
\end{proof}
\begin{Cor}[Case 4: $\eta = tw', \delta = tw$]
\begin{equation*}
[wt^{-1}]_f [tw]_g = \tau[1]_{fg} + [w]_{fg}^1.
\end{equation*}
\end{Cor}
\begin{Thm}[Case 5: $\eta = w, \delta = wt^{-1}$]
\begin{equation*}
[w]_f [wt^{-1}]_g = \tau[t^{-1}]_{fg} + [wt^{-1}]_{fg}^1.
\end{equation*}
\end{Thm}
\begin{proof}
The reasoning is now standard: We first check that
\begin{equation*}
\mathscr{P}w\mathscr{P}wt^{-1}\mathscr{P} = \mathscr{P}w\mathscr{P}w\mathscr{P}t^{-1}\mathscr{P} \subset 
\mathscr{P}t^{-1}\mathscr{P} \cup \mathscr{P}wt^{-1}\mathscr{P} \cup \bigcup_{x\in Y} \mathscr{P}x\mathscr{P},
\end{equation*}
where $Y\cap W = \emptyset$ (this follows from the observation that multiplication by
$t^{-1}$ does respect the property of being contained in $W$ or not). Now, in 
$\mathscr{H}_{\mathbb{Z}}(G, \mathscr{P}, 1)$, we have $[w][wt^{-1}] = \lambda[t^{-1}] + \ldots$ 
with $\lambda = \#(\mathscr{P}w\mathscr{P} \cap t^{-1}\mathscr{P}tw\mathscr{P})\backslash \mathscr{P}$. We use that $\mathscr{P}$ is stable under 
conjugation with $t$, so $\lambda = q^{(n^2)}$, and we can use the same reasoning as 
above to establish the first summand as $[t^{-1}]_{\tau fg}$.\\
For the second summand, we start with $k_1 = \alpha_x, k_2 = \beta_y$ and have
to figure out when $wt^{-1}k_2^{-1}twk_1^{-1}w \in \mathscr{P}$. This boils down to the condition
\begin{equation*}
\left( \begin{smallmatrix}
x & \textbf{1} \\ \textbf{1}+xy & y
\end{smallmatrix}\right)\in \mathscr{P}.
\end{equation*}
We get $[wt^{-1}]_{fg}^1$.

\end{proof}
\begin{Cor}[Case 6: $\eta = t^{-1}w', \delta = w'$]
\begin{equation*}
[t^{-1}w']_f [w']_g = \tau[t^{-1}]_{fg} + [t^{-1}w']_{fg}^1.
\end{equation*}
\end{Cor}
\begin{proof}
This follows from the above when writing
\begin{equation*}
[t^{-1}w']_f [w']_g = [wt^{-1}]_f [t]_g[wt^{-1}]_1 = [w]_{fg} [wt^{-1}]_1.
\end{equation*}
\end{proof}
The following is proved completely analogously:
\begin{Thm}[Case 7: $\eta = tw, \delta = w$]
\begin{equation*}
[tw]_f [w]_g = \tau[t]_{fg} + [tw]_{fg}^1.
\end{equation*}
\end{Thm}
\begin{Cor}[Case 8: $\eta = w', \delta = w't$]
\begin{equation*}
[w']_f  [w't]_g = \tau[t]_{fg} + [w't]_{fg}^1.
\end{equation*}
\end{Cor}
\section{Tensor decomposition}\label{tensor_dec}
\subsection{Twisted tensor products}\label{twistedTensorprods}
For our decomposition of the Hecke algebra we need a mild generalization of the twisted tensor 
product. This is a construction originating from the theory of quantum groups and treated 
e. g. in \cite{twisted}.

For this subsection, let $A, B$ be associative $R$-algebras over some
commutative unital ring $R$. Denote the multiplication maps by $\mu_A: A\times A\rightarrow A$ and
$\mu_B: B\times B \rightarrow B$. Let $S$ be an associative unital $R$-algebra
together with $R$-algebra homomorphisms $S\rightarrow A, S\rightarrow B$.
This allows us to regard $A$ and $B$ as $S$-$S$-bimodules 
in a way which is compatible with their structures as $R$-algebras.
\begin{Def}[Twisting map]
A twisting map (over $S$) is an $S$-linear map
\begin{equation*}
\psi: B\otimes_S A \rightarrow A\otimes_S B 
\end{equation*}
which fulfills $\psi(1_B\otimes a) = a\otimes 1_B$ and $\psi(b\otimes 1_A) = 1_A \otimes b$.
\end{Def}
If $\psi$ is a twisting map, 
the multiplication $\mu_{\psi} := (\mu_A \otimes \mu_B)\circ (\operatorname{id}_A \otimes \psi \otimes \operatorname{id}_B)$ makes 
$A\otimes_S B$ into an $R$-algebra and $S$-$S$-bimodule denoted by $A\overset{\psi}{\otimes}_S B$.
We can write the product explicitly as
\begin{equation*}
(a_1 \otimes b_1) \ast (a_2 \otimes b_2) = \sum_i (a_1 \ast_A a_2^i) \otimes (b_1^i \ast_B b_2)
\end{equation*}
if $\psi(b_1\otimes a_2) = \sum_i a_2^i \otimes b_1^i$.

\begin{Rem}
The question whether $A\overset{\psi}{\otimes}_S B$ is associative can easily be translated into 
a property of the map $\psi$, cf. sections 2.3-2.4 
of \cite{twisted}. In our applications we will be constructing $R$-algebra isomorphisms
\begin{equation*}
A\overset{\psi}{\otimes}_S B \overset{\cong}{\longrightarrow} C
\end{equation*}
where $C$ is known to be associative. This is why we do not have to check associativity separately for 
the twisted tensor products we are about to consider. 
\end{Rem}

\begin{Prop}\label{iteratedproduct}
Let $R, S, A, B$ as above and consider twisting maps $\psi_1: S\otimes_R A \rightarrow A\otimes_R S$ (over $R$) and 
$\psi_2: B \otimes_S (A\otimes_R^{\psi_1}S) \rightarrow (A\otimes_R^{\psi_1}S) \otimes_S B$ (over $S$). Then there is an 
isomorphism of $R$-algebras 
\begin{equation}\label{ejogwhi3h4903gh0394hgh40}
 A\otimes_R^{\psi_3} B \cong (A\otimes_R^{\psi_1}S) \otimes_S^{\psi_2} B,
\end{equation} 
where $\psi_3 = \gamma \circ \psi_2 \circ (\operatorname{id}_B \otimes \iota_A)$. Here, we 
made use of the maps
\begin{equation*}
\iota_A : A\rightarrow A\otimes_R^{\psi_1} S \qquad a\mapsto a\otimes 1
\end{equation*}
and
\begin{equation*}
\gamma: (A\otimes_R^{\psi_1}S) \otimes_S B \rightarrow A\otimes_R B \qquad \sum_i a_i\otimes s_i \otimes b_i 
\mapsto \sum_i a_i \otimes s_i.b_i.
\end{equation*} 
\end{Prop}
\begin{proof}
Clearly, $\psi_3$ defines a twisting map over $R$. Moreover
it is clear that the assignment $f: \sum_i a_i \otimes b_i \mapsto \sum_i a_i \otimes 1 \otimes b_i$ establishes 
the isomorphism (\ref{ejogwhi3h4903gh0394hgh40}) on the level of $R$-modules. Thus we have to check that $f$ commutes
with the algebra multiplication, which boils down to the formal calculation
\begin{equation*}
(f\times f) \circ \mu_{\psi_3} = \mu_{\psi_2}\circ (f\times f).\qedhere
\end{equation*}
\end{proof}

\subsection{Decomposing $\mathscr{H}(G,\mathscr{P},V)$}
\begin{Def}
For $1\leq \tau < \ell$, denote by $R[T]^{\tau}$ the $R$-algebra with underlying $R$-module $R[T]$, but modified 
multiplication defined by
\begin{equation*}
T^a\ast T^b = \begin{cases}
\tau T^{a+b} + T^{a+b+1} &\text{if both $a$ and $b$ are odd;}\\
T^{a+b} & \text{otherwise.}
\end{cases}
\end{equation*}
\end{Def}
\begin{Lem}
$R[T]^{\tau}$ is Euclidean (and, hence, a principal ideal domain).
\end{Lem}
\begin{proof}
First, consider the ring homomorphism
\begin{equation*}
E: R[X, Y]\rightarrow R[T] \qquad X\mapsto T, Y\mapsto T^2.
\end{equation*}
$E$ is clearly surjective, and we want to show that $\operatorname{ker}(E) = (f_0)$ with $f_0 = X^2 - Y(X+\tau)$.
The inclusion ``$\supset$'' is clear, so let $f$ be an element of the kernel. As $E(f_0)=0$, we can successively
replace the symbol $X^2$ in the expression of $f$ by $Y(X+\tau)$, and via this process we will end up with 
a polynomial of the form $f' = \sum_i a_i Y^i + X\sum_i b_iY^i$ such that $E(f) = E(f')$. But we have
\begin{equation*}
E(f') = \sum_i (a_i T^{2i} + b_i T^{2i+1}),
\end{equation*}
so if this expression vanishes we must have $f' = 0$. We conclude that any $f\in \operatorname{ker}(E)$ can be transformed into
the zero polynomial by replacing $X^2$ by $Y(X+\tau)$, i. e. $f\in (f_0)$. Hence
\begin{equation*}
R[X, Y]/(f_0) \cong R[T]^{\tau}.
\end{equation*}
Now, we can embedd
\begin{equation*}
R[X, Y]/(f_0) \hookrightarrow R(X)\qquad X\mapsto X, Y\mapsto \frac{X^2}{X+\tau}.
\end{equation*}
In order to determine the image of this embedding, we write $X^2/(X+\tau) = X -\tau + \tau(X+\tau)^{-1}$, so we get an isomorphism
$R[X, Y]/(f_0) \cong R[X]_{(X+\tau)}$ (localization at $X+\tau$).
As the property of being Euclidean is preserved by localization, the statement follows.
\end{proof}
\begin{Def}
Denote by $\mathscr{H}^{\dagger}_R(G, \mathscr{P}, V)$ the subalgebra of
$\mathscr{H}_R(G, \mathscr{P}, V)$ spanned by all elements of the form $[\eta]^a$
with $a\in \mathbb{N}_0$ and $\eta\in W$ (where, in order to make sense, $a$ must
be even (resp. odd) if $\eta$ is diagonal (resp. non-diagonal)).
\end{Def}
Let $\mathscr{K} = \operatorname{GL}_{n}(\mathcal{O})$ denote the maximal open compact subgroup
of $G$.
\begin{Def}
By $\mathscr{H}^{\dagger}_R(\mathscr{K}, \mathscr{P}, V)$ we denote the intersection 
$\mathscr{H}^{\dagger}_R(G, \mathscr{P}, V)\cap \mathscr{H}_R(\mathscr{K}, \mathscr{P}, V)$.
For $\mathcal{P}$ the standard parabolic subgroup of $\mathcal{G}$
with Levi component $\mathcal{M}$, define the Hecke algebra $\mathscr{H}_R(\mathcal{G}, \mathcal{P}, V)$
consisting of maps $\mathcal{G}\rightarrow V$ fulfilling a bi-invariance property analogous to
(\ref{HeckeChar}) (see  also section \ref{Requirements}).
Then, because of the obvious isomorphism 
\begin{equation*}
\mathscr{H}_R(\mathscr{K}, \mathscr{P}, V) \cong \mathscr{H}_R(\mathcal{G}, \mathcal{P}, V), 
\end{equation*}
we can equally well define 
$\mathscr{H}^{\dagger}_R(\mathscr{K}, \mathscr{P}, V)$ as the subalgebra of 
$\mathscr{H}_R(\mathcal{G}, \mathcal{P}, V)$ spanned by all 
elements of the form $[\eta]^a$ with $\eta\in \{1,w\}, a \in \mathbb{N}_0$. This motivates the notation
\begin{equation*} 
\mathscr{H}^{\dagger}_R(\mathcal{G}, \mathcal{P}, V)
= \mathscr{H}^{\dagger}_R(\mathscr{K}, \mathscr{P}, V).
\end{equation*}
\end{Def}
\begin{Def}[Chararacteristic polynomial]
It is clear that the assignment $\sum_I r_i T^i \mapsto
\sum_I r_i[w^i]^i$ defines an $R$-algebra homomorphism
\begin{equation*}
R[T]^{\tau} \rightarrow
\mathscr{H}_R(\mathcal{G}, \mathcal{P}, V) = \operatorname{End}_{\mathcal{G}}\left( \operatorname{ind}_{\mathcal{P}}^{\mathcal{G}}(V)\right)
\end{equation*}
with image 
\begin{equation*}
\mathscr{H}^{\dagger}_R(\mathcal{G}, \mathcal{P}, V) = \,\langle [w^a]^a \,|\,a\in \mathbb{N}_0\rangle_R.
\end{equation*}
We define $\digamma = \digamma_{R, q, k}$ to be the minimal-degree monic polynomial
which generates the kernel of this surjection. It is clear that $\digamma$ depends only on the finite group data $(R, q, k)$.
\end{Def}
\begin{Ex}\label{ExbeforeConj}
If $\ell$ divides $q-1$, we have $\tau = 1$ and we get $\digamma = T^2$ (and, consequently,
an isomorphism between $R[T]^{\tau}$ and the dual numbers $R[\epsilon]$) if $T^{\ast}.V \neq 0$ and 
$\digamma = T$ (and $R[T]^{\tau} \cong R$) if $T^{\ast}.V = 0$.
\end{Ex}
It is a natural question to ask if and how these polynomials are related for different choices of 
the data $(R, q, k)$. In particular, one could conjecture that they depend only on $R$ and $q^k$, 
i. e. that $\digamma_{R, q, ab} = \digamma_{R, q^a, b}$ for any choice of $a,b\geq 1$.
\begin{Lem}[Failure of Commutativity]
In $\mathscr{H}_R(G, \mathscr{P}, V)$ we have $[1]_f [\eta]^a = [\eta]^a [1]_f$ and 
\begin{itemize}
\item 
\begin{equation*}
[w]_f [\eta]^a = [w\eta w]^a[w]_f \qquad \text{ if }\eta = t^{2b}w'\ldots w\text{ or } \eta = t^{2b+1}w\ldots w;
\end{equation*}
\item
\begin{equation*}
[w]_f [\eta]^a = (\tau [w\eta w]^a + [\eta w]^{a+1})[w]_f
\end{equation*}
\begin{equation*}
= \tau [w\eta w]^a [w]_f + [\eta]^a [1]_f^1
\qquad\text{ if }\eta = t^{2b}w\ldots w\text{ or } \eta = t^{2b+1}w'\ldots w;
\end{equation*}
\item
\begin{equation*}
[w]_f [\eta]^a = \tau^{-1}([w\eta w]^a - [w\eta ]^{a+1})[w]_f  
\end{equation*}
\begin{equation*}
= \tau^{-1}[w\eta w]^a([w]_f - [1]_{f}^1)
\qquad\text{ if }\eta = t^{2b}w'\ldots w'\text{ or } \eta = t^{2b+1}w\ldots w';
\end{equation*}
\item
\begin{equation*}
[w]_f [\eta]^a = ([w\eta w]^a - [w\eta ]^{a+1} + \tau^{-1}[\eta w]^{a+1} - \tau^{-1}[\eta]^{a+2})[w]_f 
\end{equation*}
\begin{equation*}
= [w\eta w]^a[w]_f + ([\eta]^a - [w\eta w]^a)[1]_{f}^1
\qquad\text{ if }\eta = t^{2b}w\ldots w'\text{ or } \eta = t^{2b+1}w'\ldots w'.
\end{equation*}
\end{itemize}
\end{Lem}
\begin{proof}
Straight-forward calculation.
\end{proof}

\begin{Prop}
A diagonal element $\delta_{x,y} := \operatorname{diag}(\varpi^x, \ldots,\varpi^x, \varpi^y,\ldots, \varpi^y) \in W$ is of
the form $t^{2b}w'\ldots w$ or $t^{2b+1}w\ldots w$ if and only if $x\leq y$
and of the form $ t^{2b}w\ldots w'$ or $t^{2b+1}w'\ldots w'$ iff $x\geq y$.
\end{Prop}
\begin{proof}
It is an easy calculation that 
\begin{itemize}
\item
$\delta_{x,y} = t^{y+x}(w'w)^{(y-x)/2}$ if $y\geq x$ and $y-x \equiv 0 \operatorname{mod} 2$;
\item
$\delta_{x,y} = t^{y+x}w(w'w)^{(y-x-1)/2}$ if $y\geq x$ and  $y-x \equiv 1 \operatorname{mod} 2$;
\item
$\delta_{x,y} = t^{y+x}(ww')^{(x-y)/2}$ if $y\leq x$ and $y-x \equiv 0 \operatorname{mod} 2$;
\item
$\delta_{x,y} = t^{y+x}(w'w)^{(x-y-1)/2}w'$ if $y\leq x$ and  $y-x \equiv 1 \operatorname{mod} 2$.
\end{itemize}\qedhere
\end{proof}
\begin{Thm}
The assignment
\begin{equation*}
\psi: R[T]^{\tau}_{/\digamma}\otimes_R R[\mathbb{Z}^2] \rightarrow
R[\mathbb{Z}^2] \otimes_R  R[T]^{\tau}_{/\digamma}
\end{equation*}
\begin{equation*}
\left(\sum_i r_i T^i\right)\otimes (\alpha, \beta) \mapsto
\begin{cases}
(\alpha, \beta) \otimes (\sum_{2|i} r_i T^i) + (\beta, \alpha) \otimes (\sum_{2\not\;|\,i} r_i T^i) & \text{ if }\alpha \leq \beta;\\
(\alpha, \beta) \otimes (\sum_{2|i} (r_i + r_{i-1}) T^i) + (\beta, \alpha) \otimes (\sum_{2\not\;|\,i} r_i (T^i- T^{i+1}))
& \text{ if } \alpha \geq \beta;
\end{cases}
\end{equation*}
defines a twisting map. (In this formula, the symbol $r_{-1}$ is to be interpreted as $0$.)
Moreover, there is an isomorphism of $R$-algebras
\begin{equation}\label{first_iso}
R[\mathbb{Z}^2]\otimes_R^{\psi}R[T]^{\tau}_{/\digamma}\cong
\mathscr{H}^{\dagger}(G, \mathscr{P}, V)
\end{equation}
induced by the assignment
\begin{equation*}
E:(\alpha, \beta)\otimes\left(\sum_i r_i T^i\right) \mapsto  \delta_{\alpha, \beta}\ast\left(\sum_i r_i [w^i]^i\right).
\end{equation*}
\end{Thm}
\begin{proof}
It is clear that $\psi$ defines a twisting map. It follows from the failure-of-commutativity lemma
that $E$ gives rise to an $R$-algebra map. Thus, the isomorphism is established as soon as we find a set-theoretic
inverse mapping $G$ of $E$. For this, consider
\begin{equation*}
G: [\eta]^a \mapsto \begin{cases}
(x,y) \otimes T^a & \text{ if } \eta = \delta_{x,y};\\
(x,y) \otimes T^a & \text{ if } \eta = \delta_{x,y}w \text{ and } x\geq y;\\
(x,y) \otimes \tau^{-1}(T^a - T^{a+1}) & \text{ if } \eta = \delta_{x,y}w \text{ and } x < y.
\end{cases}
\end{equation*}
It is now straightforward to check that $E\circ G = 1$ and $G\circ E = 1$.
\end{proof}
\begin{Cor}
The isomorphism (\ref{first_iso}) is an isomophism of $R[T]^{\tau}_{/ \digamma}$-bimodules via
\begin{equation*}
\left(\sum_i r_i T^i \right) \ast\left( (\alpha, \beta) \otimes h\right) \ast g 
\end{equation*}
\begin{equation*}
= \begin{cases}
(\alpha, \beta) \otimes (\sum_{2|i} r_i T^i)hg + (\beta, \alpha) \otimes (\sum_{2\not\;|\,i} r_i T^i)hg & \text{ if }\alpha \leq \beta;\\
(\alpha, \beta) \otimes (\sum_{2|i} (r_i + r_{i-1}) T^i)hg + (\beta, \alpha) \otimes (\sum_{2\not\;|\,i} r_i (T^i- T^{i+1}))hg
& \text{ if } \alpha \geq \beta;
\end{cases}
\end{equation*}
where $\sum_i r_i T^i, h, g \in R[T]^{\tau}_{/ \digamma}$ and $(\alpha, \beta) \in R[\mathbb{Z}^2]$.
\end{Cor}
\begin{Cor}
\begin{equation*}
\mathscr{H}^{\dagger}(\mathcal{G}, \mathcal{P}, V) \cong R[T]^{\tau}_{/\digamma}.
\end{equation*}
\end{Cor}
The main decomposition is now
\begin{Thm}
The assignment
\begin{equation*}
[1]_f\otimes [\eta]^a \mapsto [\eta]^a \otimes [1]_f;
\end{equation*}
\begin{equation*}
[w]_f \otimes [\eta]^a \mapsto \begin{cases}
[w\eta w]^a \otimes [w]_f & \text{ if }\eta = t^{2b}w'\ldots w \text{ or } t^{2b+1}w\ldots w;\\
\tau[w\eta w]^a \otimes [w]_f + [\eta]^a\otimes [1]_{f}^1& \text{ if }\eta = t^{2b}w\ldots w \text{ or } t^{2b+1}w'\ldots w;\\
\tau^{-1}[w\eta w]^a \otimes ([w]_f -[1]_{f}^1)& \text{ if }\eta = t^{2b}w'\ldots w' \text{ or } t^{2b+1}w\ldots w';\\
[w\eta w]^a \otimes ([w]_f -[1]_{f}^1) + [\eta]^a\otimes [1]_{f}^1& \text{ if }\eta = t^{2b}w\ldots w' \text{ or } t^{2b+1}w'\ldots w';\\
\end{cases}
\end{equation*}
defines a twisting map
\begin{equation*}
\Psi:
\mathscr{H}_R(\mathcal{G}, \mathcal{P},V) \otimes_{\mathscr{H}^{\dagger}_R(\mathcal{G}, \mathcal{P},V)} \mathscr{H}^{\dagger}_R(G, \mathscr{P},V)
\rightarrow
\mathscr{H}^{\dagger}_R(G, \mathscr{P}, V)\otimes_{\mathscr{H}^{\dagger}_R(\mathcal{G}, \mathcal{P},V)} \mathscr{H}_R(\mathcal{G}, \mathcal{P},V).
\end{equation*}
This gives rise to an isomorphism of $\mathscr{H}^{\dagger}_R(\mathcal{G}, \mathcal{P},V)$- (and, hence, $R$-)algebras
\begin{equation*}
\mathscr{H}^{\dagger}_R(G, \mathscr{P},V)\otimes^{\Psi}_{\mathscr{H}^{\dagger}_R(\mathcal{G}, \mathcal{P},V)} \mathscr{H}_R(\mathcal{G}, \mathcal{P},V)
\cong \mathscr{H}_R(G, \mathscr{P},V).
\end{equation*}
\end{Thm}
\begin{proof} 
Because of the obvious $R$-linearity of the given assignment, the verification that $\Psi$ is well-defined and linear 
can be done case-by-case (depending
on the structure of $\eta$). So, let $\eta$ be as in the first case above, i. e. 
$\eta = t^{2b}w'\ldots w \text{ or } t^{2b+1}w\ldots w$. It suffices to compare the following
expressions:
\newcommand{\msout}[1]{\text{\sout{\ensuremath{#1}}}}
\begin{equation*}
{\Psi([1]_f[w]^b \otimes [\eta]^a) \overset{?}{=} \Psi([1]_f \otimes [w]^b[\eta]^a) \quad \text{and}\quad
\Psi([w]_f[w]^b \otimes [\eta]^a) \overset{?}{=} \Psi([w]_f \otimes [w]^b[\eta]^a)}.
\end{equation*}
After evaluation, the first comparison is between $[w\eta w]^a \otimes [w]_{f}^b$ and
${[w\eta]^{a+b}\otimes [1]_f}$
, which coincides as we have $[w\eta]^{a+b} = [w\eta w]^a [w]^b$.
The second comparison, after evaluation, is between $\tau[\eta]^a \otimes [1]_{f}^b+
[w\eta w]^a\otimes [w]_{ f}^{b+1}$ and $\tau[\eta w]^{a+b}\otimes [w]_f + [w\eta]^{a+b}\otimes
[1]_{f}^1$.
Equality follows from the identities 
\begin{equation}\label{simplifications}
\tau[1]_{f}^1 = ([w]^1 - [1]^2)[w]_f\text{ and }[\eta]^c([w]^1 - [1]^2) = \tau[\eta w]^{c+1}.
\end{equation}
For the 
$\mathscr{H}^{\dagger}_R(\mathcal{G}, \mathcal{B},V)$-linearity, it is clearly sufficient to 
prove that $[w]^1 \circ \Psi = \Psi \circ [w]^1$.
We have
\begin{equation*}
[w]^1 \circ \Psi([1]_f \otimes [\eta]^a) = [w]^1\left( [\eta]^a \otimes [1]_f\right) =
[w\eta]^{a+1} \otimes [1]_f. 
\end{equation*}
On the other hand, 
\begin{equation*}
\Psi \circ [w]^1([1]_f \otimes [\eta]^a) = \Psi\left([w]_{f}^1\otimes [\eta]^a\right) = 
\Psi\left([1]_{f}\otimes [w\eta]^{a+1}\right) =
[w\eta]^{a+1} \otimes [1]_f.
\end{equation*}
Moreover, we have
\begin{equation*}
[w]^1 \circ \Psi([w]_f \otimes [\eta]^a) = [w]^1\left([w\eta w]^a \otimes [w]_f\right) = 
(\tau[\eta w]^{a+1} + [w\eta w]^{a+2})\otimes [w]_f. 
\end{equation*}
On the other hand,
\begin{equation*}
\Psi \circ [w]^1([w]_f \otimes [\eta]^a) = \Psi\left( (\tau[1]_{f}^1 + [w]_{f}^2) \otimes [\eta]^a\right)
= \tau[\eta]^a \otimes [1]_{f}^1 + [w\eta w]^a \otimes [w]_{f}^2.
\end{equation*}
Here, we used the identity in (\ref{simplifications}) again. The other three cases are checked similarly.\\
That $\Psi(1\otimes X) = X\otimes 1$ and $\Psi(Y\otimes 1) = 1\otimes Y$ is obvious. (Remark 
that $\eta = 1$ is included in both the first and the last case in the statement of the theorem.)\\
Now the desired isomorphism is given
by
\begin{equation*}
E: [\eta]^a \otimes \left( [1]_f + [w]_g\right) \mapsto [\eta]^a \left( [1]_f + [w]_g\right).
\end{equation*}
This defines an $R$-algebra homomorphism by the failure-of-commutativity lemma. In order to see that it
is bijective it is sufficient to give a set-theoretic inverse map. Similar to the last proof, we 
consider
\begin{equation*}
G: [\eta]_f \mapsto \begin{cases}
[\eta]^0 \otimes [1]_f& \text{ if $\eta$ diagonal};\\
[\eta w]^0 \otimes [w]_f & \text{ if $\eta$ non-diagonal, ends on }w;\\
\tau^{-1}[\eta w]^0 \otimes ([w]_f -[1]_{f}^1) & \text{ if $\eta$ non-diagonal, does not end on }w.
\end{cases}
\end{equation*}
(The third case includes $\eta = t^{2b+1}$.) It is clear that $E\circ G = 1$. To check that
$G\circ E = 1$ on symbols of the form $[\eta]^a\otimes [1]_f$ involves checking three
cases (depending on the form of $\eta$, as distinguished in the definition of $G$), which
is straightforward. To check that
$G\circ E = 1$ on symbols of the form $[\eta]^a\otimes [w]_f$ boils down to checking
four cases, which, again, poses no difficulty. 
\end{proof}

A variation of this decomposition is:
\begin{Cor}\label{MainResult}
The assignment
\begin{equation*}
[1]_f\otimes (\alpha, \beta) \mapsto (\alpha, \beta) \otimes [1]_f;
\end{equation*}
\begin{equation*}
[w]_f \otimes (\alpha, \beta) \mapsto \begin{cases}
(\beta, \alpha) \otimes [w]_f & \text{ if }\alpha \leq \beta;\\
(\alpha, \beta) \otimes [1]_{f}^1 + (\beta, \alpha)\otimes ([w]_f- [1]_{f}^1)& \text{ if }\alpha \geq \beta;
\end{cases}
\end{equation*}
defines a twisting map
\begin{equation*}
\zeta:
\mathscr{H}_R(\mathcal{G}, \mathcal{P},V) \otimes_R R[\mathbb{Z}^2] \rightarrow
R[\mathbb{Z}^2]\otimes_R \mathscr{H}_R(\mathcal{G}, \mathcal{P},V).
\end{equation*}
This gives rise to an isomorphism of $R$-algebras
\begin{equation*}R[\mathbb{Z}^2]\otimes_R^{\zeta} \mathscr{H}_R(\mathcal{G}, \mathcal{P},V)
\cong \mathscr{H}_R(G, P,V).
\end{equation*}
\end{Cor}
\begin{proof}
By the results above, we know
that 
\begin{equation*}
\mathscr{H}(G, P, V) = \left( R[\mathbb{Z}^2]\otimes_R^{\psi}R[T]^{\tau}_{/\digamma} \right)
\otimes_{R[T]^{\tau}_{/\digamma}}^{\Psi} 
\mathscr{H}_R(\mathcal{G}, \mathcal{P},V).
\end{equation*}
Therefore, the claim follows from Proposition \ref{iteratedproduct}.
\end{proof}
\begin{Ex}[Scalar Iwahori-Hecke algebra]
We can recover the scalar case (i. e. where $G=\operatorname{GL}_2, k=1$ and $V$ is the trivial
representations) as follows: Let $W=S_2$ be the Weyl group of $G$, then it is easy to see that $\mathscr{H}(\mathcal{G}, \mathcal{B}, V)$
is isomorphic to the combinatorial Hecke algebra $\mathscr{H}(W, q)$ (cf. \cite{VigLivre}, I.3.13, 
where this algebra is referred to as $H^0_R(2, q)$). As we are interested in the modular case $\ell | \#\operatorname{GL}_2(q)$, we 
have either $\ell | (q-1)$ or $\ell |(q+1)$. Let us assume that we are in the first case, so that $\mathscr{H}(W, q)$ is isomorphic the group algebra 
of $W$. If we assume moreover that $q > 2$, we see that $T^{\ast}|V = 0$ and hence the map 
$\zeta$ from the above corollary comes from the action
\begin{equation*}
\xi: S_2 \rightarrow \operatorname{Aut}(\mathbb{Z}^2) \qquad
1.(x,y) = (x,y) \text{ and } w.(x,y) = (y,x).
\end{equation*}
Let $I\subset G$ denote the Iwahori subgroup and $\tilde{W}$ the affine Weyl group of $G$. We conclude that
\begin{equation*}
\mathscr{H}(G, I, 1) \cong R[\mathbb{Z}^2]\overset{\zeta}{\otimes}_R R[S_2] \cong
R[\mathbb{Z}^2 \rtimes_{\xi} S_2] = R[\tilde{W}],
\end{equation*}
what matches with the classical description of $\mathscr{H}(G, I, 1)$ 
(\cite{VigLivre}, I.3.14). In the case $\ell | q+1$ one gets correspondingly an isomorphism between
$\mathscr{H}(G, I, 1)$ and the combinatorial (affine) Hecke algebra $H(\tilde{W}, -1)$.
%
%
\end{Ex}
\section{Relating simple blocks of different groups}
Retain the assumptions on $F$ and $R$ from the introduction and fix 
two numbers $k, m\geq 1$ and denote $G_1 = \operatorname{GL}_{2km}(F), G_2 = \operatorname{GL}_{2k}(F^m)$, where $F^m$
denotes the unramified extension of $F$ of degree $m$. Moreover, let $P_i\subset G_i$ be the standard
parabolic subgroups characterized by having a Levi decomposition
\begin{equation*}
P_i = M_i U_i \text{ with } M_i \cong H_i \times H_i,
\end{equation*}
where $H_1 = \operatorname{GL}_{km}(F)$ and $H_2 = \operatorname{GL}_{k}(F^m)$.

Consider two irreducible supercuspidal level-$0$ 
representations $\pi_i\in \operatorname{Rep}_R(H_i)$ (for $i=1,2$) and 
denote $\pi_i^2 = \pi_i \boxtimes \pi_i \in \operatorname{Rep}_R(M_i)$.
Having the banal situation in mind, it is natural to consider the following conjecture:
\begin{Conj}\label{mainConj}
There is an equivalence of categories
\begin{equation*}
\mathfrak{R}^{[M_1, \pi_1^2]_{G_1}}_R(G_1) \cong
\mathfrak{R}^{[M_2, \pi_2^2]_{G_2}}_R(G_2).
\end{equation*}
\end{Conj}
\begin{Rem}The case where $k=1$ is the most interesting and also the most general, as the proof for an arbitrary
$k$ can be given by applying the conjecture two times in the $k=1$-version.
\end{Rem}
\subsection{Connection with a conjecture on finite groups}\label{Requirements}
For this final subsection (where we restrict to the $k=1$ case as suggested in the above remark), 
consider the following (self-contained but slightly differing) notation:
\begin{itemize}
\item $R$ denotes an algebraically closed field of positive characteristic $\ell$;
\item $q$ denotes a power of a prime different from $\ell$;
\item $m\in \mathbb{N}_+$ and $\mathcal{G}_1 = \operatorname{GL}_{2m}(q), \mathcal{H}_1 = \operatorname{GL}_{m}(q), \mathcal{G}_2 = \operatorname{GL}_{2}(q^m),
\mathcal{H}_2 = \operatorname{GL}_{1}(q^m) = \mathbb{F}_{q^m}^{\times}$;
\item $w_1\in \mathcal{G}_1$ denotes the permutation matrix corresponding to $(1, m+1)(2, m+2)\ldots (m, 2m)\in S_{2m}$
and $w_2 = \left(\begin{smallmatrix}& 1\\ 1\end{smallmatrix}\right)\in \mathcal{G}_2$ denotes the permutation matrix corresponding to $(1, 2)\in S_{2}$;
\item $\mathcal{P}_i$ denotes the standard parabolic subgroup of $\mathcal{G}_i$ characterized by admitting
a Levi decomposition with Levi factor $\mathcal{M}_i = \mathcal{H}_i \times \mathcal{H}_i$;
\item $\textbf{i}_i$ denotes the Harish-Chandra induction functor from representations of $\mathcal{M}_i$ to 
representations of $\mathcal{G}_i$;
\item $\pi_2=R$ denotes the trivial character of $\mathcal{H}_2$ and $\pi_1$ denotes some supercuspidal irreducible representation of $\mathcal{H}_1$ over $R$. 
$\pi_i^2$ denotes $\pi_i \boxtimes \pi_i$ as a representation of $\mathcal{M}_i$. 
\end{itemize}
For $V_i$ a representation of $\mathcal{M}_i$, the elements of the Hecke algebra
$\mathscr{H}_R(\mathcal{G}_i, \mathcal{M}_i, V_i) =
\operatorname{End}_{\mathcal{G}_i}(\textbf{i}_i(V_i))$ can be characterized (by \cite{VigLivre}, I.8.5)
as maps $\varphi: G\rightarrow \operatorname{End}_R(V_i)$ with a similar bi-equivariance condition
as in (\ref{HeckeChar}). Thus, denote by $T^{\ast}_{V_i}$ the element of $\mathscr{H}_R(\mathcal{G}_i, \mathcal{M}_i, V_i)$
which is supported on $\mathcal{P}_i w_i \mathcal{P}_i$ and determined by 
\begin{equation*}
w_i \overset{T^{\ast}_{V_i}\,}{\longmapsto} \sum_{g\in \mathcal{H}_i} \begin{pmatrix}g& \\ & -g^{-1}
\end{pmatrix}.
\end{equation*}
\begin{Conj}\label{BroueConj}
There exist representations ${V}_i \in \operatorname{Rep}_R(\mathcal{M}_i)\;\, (i=1,2)$ fulfilling
\begin{itemize}
\item ${V}_i^{\tt{ss}}\cong \oplus \pi_i^2$ (where $\texttt{ss}$ denotes the semi-simplification), i. e. all JH-constituents of $V_i$ are
isomorphic to $\pi_i^2$;
\item ${V}_i$ is projective and isomorphic to its contragredient;
\end{itemize}
such that there exists an isomorphism of $R$-algebras
\begin{equation*}
\varphi: \mathscr{H}_R(\mathcal{G}_1, \mathcal{M}_1, {V}_1) \rightarrow
\mathscr{H}_R(\mathcal{G}_2, \mathcal{M}_2, {V}_2)
\end{equation*}
which maps $T^{\ast}_{{V}_1}$ to $T^{\ast}_{{V}_2}$.
\end{Conj}
If Conjecture \ref{BroueConj} holds true, it seems reasonable to expect that the self-dual pro-generators constructed
in Theorem \ref{progenerator} will do the job. However, independent of the origin of these $V_i$ and as outlined
in the introduction, we can make use of the results of the previous chapter (in particular of Corollary
\ref{MainResult}) to deduce the following:
\begin{Cor}
Conjecture \ref{BroueConj} implies 
Conjecture \ref{mainConj}. Moreover, it implies that the conjectured identity stated below Example \ref{ExbeforeConj} holds.
\end{Cor}

\end{document}